\documentclass[11pt,reqno]{amsart}
\usepackage{xifthen,setspace}
\usepackage{bbm}
\usepackage{pkgfile}
\usepackage{float} 
\usepackage{caption,subcaption}

\newtheorem{theorem}{Theorem}
\newtheorem{lemma}{Lemma}

\newtheorem{prop}{Proposition}

\newcommand{\p}{{\mathbb{P}}}

\newcommand{\e}{{\mathbb{E}}}

\spacing{1.025}
\allowdisplaybreaks

\title[Interaction Order Estimation in Tensor Curie-Weiss Models]{Interaction Order Estimation in Tensor Curie-Weiss Models}

\author[Mukherjee]{Somabha Mukherjee} 
\address{Department of Statistics and Data Science, National University of Singapore {\tt somabha@nus.edu.sg}}

\begin{document}
	
	\maketitle
	
	\begin{abstract} 
		In this paper, we consider the problem of estimating the interaction parameter $p$ of a $p$-spin Curie-Weiss model at inverse temperature $\beta$, given a single observation from this model. We show, by a contiguity argument, that joint estimation of the parameters $\beta$ and $p$ is impossible, which implies that estimation of $p$ is impossible if $\beta$ is unknown. These impossibility results are also extended to the more general $p$-spin Erd\H{o}s-R\'enyi Ising model. The situation is more delicate when $\beta$ is known. In this case, we show that there exists an increasing threshold function $\beta^*(p)$, such that for all $\beta$, consistent estimation of $p$ is impossible when $\beta^*(p) > \beta$, and for \textit{almost all} $\beta$, consistent estimation of $p$ is possible for $\beta^*(p)<\beta$.
	\end{abstract}
	
	\section{Introduction}
	The Ising model \cite{ising} was originally introduced in Physics as a model for ferromagnetism, and has since then found numerous interesting applications in diverse areas, such as image processing \cite{geman_graffinge}, neural networks \cite{neural}, spatial statistics \cite{spatial}, and disease mapping in epidemiology \cite{disease}. It is a discrete exponential family on the set of all binary tuples of a fixed length, with sufficient statistic given by a quadratic form, designed to capture pairwise dependence between the binary variables, arising from an underlying network structure. Unfortunately, in most real-life scenarios, pairwise interactions are not enough to explain all the complex dependencies arising in a network data, and one has to take into account higher-order interactions arising from peer-group effects. Multi-body interactions are also common in many other branches of science; for example in Chemistry \cite{multiat}, it is known that the atoms on a crystal surface do not just interact in pairs, but in triangles, quadruples and higher order tuples.
 A natural extension of the classical $2$-spin Ising model which is designed to capture higher-order  dependencies, is the $p$-spin Ising model
 \cite{barra, derrida,jaesungcw, jaesung1}, where the quadratic interaction term in the sufficient statistic is replaced by a multilinear polynomial of degree $p\ge 2$. The probability mass function of this model is given by:
	\begin{equation}\label{genising}
		\p_{\beta,p}(\bm x) := \frac{\exp\left\{\beta \sum_{1\le i_1,\ldots,i_p\le N} J_{i_1\ldots i_p} x_{i_1}\ldots x_{i_p}\right\}}{Z_N(\beta,p)} \quad \text{for}~\bm x\in \{-1,1\}^n
	\end{equation}
	where $\bm J := ((J_{i_1\ldots i_p}))_{(i_1,\ldots,i_p) \in [N]^p}$ is the \textit{interaction tensor}, $\beta >0$ is the \textit{interaction strength} (or inverse temperature in the physics literature), $p$ is the \textit{interaction order}, and $Z_N(\beta,p)$ is the normalizing constant, needed to ensure that the probabilities in \eqref{genising} add to $1$. 
 
Estimation of the parameter $\beta$ in the model \eqref{genising} assuming that the interaction order $p$ is known, has been studied exptensively in the past (see, for example, \cite{chspin,BM16,pg_sm} for the case $p=2$, and \cite{jaesung1, jaesungcw} for the case $p\ge 3$). However, to the best of our knowledge, the reverse question of estimating the interaction order $p$ given a single observation from the model \eqref{genising} even with known $\beta$, has not been addressed in the literature. This question may be quite hard for arbitrary underlying interaction tensors $\bm J$, which necessitates some convenient structural assumptions on this tensor. 

In this paper, we are going to assume that the tensor $\bm J$ has all entries equal to $N^{1-p}$, which corresponds to the $p$-spin Curie-Weiss model \cite{barra,barra2,jaesungcw} given by:
 \begin{equation}\label{cwm}
     \p_{\beta,p}(\bm x) := \frac{\exp\left\{\beta N \bar{x}^p\right\}}{Z_N(\beta,p)} \quad \text{for}~\bm x\in \{-1,1\}^n
 \end{equation}
where $\bar{x} := \frac{1}{N}\sum_{i=1}^N x_i$. Even under this structural assumption, the possibility of estimating $p$ depends on whether the interaction strength $\beta$ is known or not. We will show that consistent estimation of $p$ is impossible if $\beta$ is unknown, which will be a consequence of our argument on the impossibility of joint consistent estimation of $\beta$ and $p$. At the heart of these impossibility results is the fact that the sufficient statistic $\bar{X}$ in the model \eqref{cwm} converges to a mixture of point masses at the maximizers of a certain function $H_{\beta,p}$, and that these maximizers do not uniquely determine the tuple $(\beta,p)$. This idea is formalized by a contiguity argument between the Curie-Weiss measures and $N$-fold products of Rademacher distributions with the largest maximizer of $H_{\beta,p}$ as mean. It should be mentioned here, that the related problem of joint inestimability of $(\beta,h)$ for the classical $2$-spin Curie-Weiss model with an additional magnetic field parameter $h$ was addressed in \cite{pg_sm} using similar contiguity arguments. We also extend these impossibility results to the more general $p$-spin Erd\H{o}s-R\'enyi Ising models. 

The situation is more intricate when $\beta$ is known. In this case, we will show that there exists a strictly increasing threshold function $\beta^*(p)$, such that consistent estimation of $p$ is impossible  for $\beta^*(p)>\beta$. However, for \textit{almost all} $\beta$ (to be precise, for all but possibly countably many $\beta$), $p$ can be estimated consistently whenever $\beta^*(p) <\beta$. The question of exactly describing the exceptional set of countably many $\beta$s for which the region $\beta^*(p) < \beta$ is inestimable for $p$, is still open.
  
\section{Impossibility of Jointly Estimating $(\beta,p)$}
We start by showing that joint consistent estimation of $\beta$ and $p$ using only one sample $\bm X := (X_1,\ldots,X_N)$ from the model \eqref{cwm} is in general, impossible. Towards this, for every $m\in [0,1)$, define a set:
$$\Theta_m := \{(\beta,p)\in (0,\infty)\times D: m~\text{is the unique non-negative global maximizer of}~H_{\beta,p}\}$$
where $D$ denotes the set of all integers $\ge 3$, and
$$H_{\beta,p}(x) := \beta x^p - \frac{1}{2}\left\{(1+x)\log (1+x) + (1-x)\log (1-x)\right\}\quad \text{for}~x\in [0,1]~.$$
\begin{theorem}\label{thm01}
    For every $m\in [0,1)$ such that $|\Theta_m|\ge 2$, there does not exist any sequence of estimators (measurable functions of $\bm X$) which is consistent for $(\beta,p) \in \Theta_m$ under the model \eqref{cwm}.
\end{theorem}

The following lemma is crucial for proving Theorem \ref{thm01}.

\begin{lemma}\label{contig}
For every $m\in [0,1)$, denote by $\mu$ the distribution of a Rademacher random variable with mean $m$. Then, the product measure $\mathbb{Q} := \otimes_{i=1}^N \mu$ is contiguous to $\p_{\beta,p}$ for all $m\in [0,1)$ and all $(\beta,p)\in \Theta_m$.
\end{lemma}

\begin{proof}
    To begin with, note that on the event $E_K:= \{\sqrt{N}(\bar{X}-m) \le K\}$, we have: 
    \begin{eqnarray*}
\frac{\mathbb{Q}(\bm X)}{\p_{\beta,p}(\bm X)} &=& \frac{Z_N(\beta,p)\prod_{i=1}^N \left(\frac{1+m}{2}\right)^{(1+X_i)/2} \left(\frac{1-m}{2}\right)^{(1-X_i)/2}}{\exp(\beta N \bar{X}^p)}\\ &=& \frac{2^{-N} Z_N(\beta,p)}{\exp\left( N\left[\beta \bar{X}^p - \frac{1}{2}\left\{(1+\bar{X})\log(1+m) +(1-\bar{X})\log (1-m)\right\} \right] \right)}~.
    \end{eqnarray*}
    Now, by Lemma 3.2 and Lemma 3.4 in \cite{jaesungcw}, we have:
    $$2^{-N} Z_N(\beta,p) = e^{NH_{\beta,p}(m)}\Theta(1)~.$$
Also, by Taylor expansion, we have:
\begin{eqnarray*}
    &&\beta \bar{X}^p - \frac{1}{2}\left\{(1+\bar{X})\log(1+m) +(1-\bar{X})\log (1-m)\right\} - H_{\beta,p}(m)\\&=& H_{\beta,p}(\bar{X}) - H_{\beta,p}(m) + \frac{1}{2}\left\{(1+\bar{X})\log\left(\frac{1+\bar{X}}{1+m}\right) +(1-\bar{X})\log \left(\frac{1-\bar{X}}{1-m}\right)\right\}\\&=& O\left((\bar{X}-m)^2\right) + (\bar{X}-m)\left(\frac{1+\bar{X}}{1+m} -\frac{1-\bar{X}}{1-m}\right)= O\left((\bar{X}-m)^2\right).
\end{eqnarray*}
Hence, we have:
$$   \frac{\mathbb{Q}(\bm X)}{\p_{\beta,p}(\bm X)} = \Theta(1) e^{NO\left((\bar{X}-m)^2\right)} \le C_K$$
for some constant $C_K$.
Hence, for any event $A_N \subseteq \{-1,1\}^N$, we have:
$$
    \mathbb{Q}(\bm X\in A_N , E_K) =\e_{\p_{\beta,p}} \left[\frac{\mathbb{Q}(\bm X)}{\p_{\beta,p}(\bm X)} \mathbbm{1}_{\bm X\in A_N , E_K}\right]\le C_K \p(\bm X\in A_N , E_K),
$$
thereby giving us:
$$\mathbb{Q}(\bm X \in A_N) \le C_K \p(\bm X\in A_N) + \mathbb{Q}(E_K^c)~.$$ Now, suppose that $\p(\bm X\in A_N) \rightarrow 0$. Then, for every $K>0$, we have we have:
$$\limsup_{N\rightarrow \infty} \mathbb{Q}(\bm X \in A_N) \le \limsup_{N\rightarrow \infty} \mathbb{Q}\left(\sqrt{N}(\bar{X}-m) > K\right) = \p(N(0,1-m^2)>K).$$ Taking limit as $K\rightarrow \infty$ throughout the above inequality, we can conclude that $\mathbb{Q}(\bm X \in A_N) \rightarrow 0$, which completes the proof of Lemma \ref{contig}.   
\end{proof}

With Lemma \ref{contig} in hand, we are now ready to prove Theorem \ref{thm01}.

\begin{proof}[Proof of Theorem \ref{thm01}]
    Suppose that there exists a sequence $T_N(\bm X) \in \mathbb{R}^2$ of consistent estimators of $(\beta,p)$ on $\Theta_m$. Fixing two different points $(\beta_1,p_1)$ and $(\beta_2,p_2)\in \Theta_m$, we can construct disjoint neighborhoods $\mathcal{B}_1$ and $\mathcal{B}_2$ around them, respectively. By consistency of $T_N(\bm X)$ on $\Theta_m$, we have:
    $$\p_{\beta_i,p_i}(T_N(\bm X) \in \mathcal{B}_i) \rightarrow 1\quad\text{for}~ i=1,2.$$ By Lemma \ref{contig}, we thus have:
    $$\mathbb{Q}(T_N(\bm X) \in \mathcal{B}_i) \rightarrow 1\quad\text{for}~ i=1,2,$$
    which contradicts the facts that $\mathcal{B}_1$, $\mathcal{B}_2$ are disjoint, and $\mathbb{Q}$ is a probability measure.
\end{proof}

\begin{remark}\label{rem1}
    Our argument implies that consistent estimation of $p$ is impossible if $\beta$ is unknown. For, if there were such a consistent estimator $\hat{p} :=\hat{p}(\bm X)$, then we could choose $(\beta_1,p_1)$ and $(\beta_2,p_2)$ from some $\Theta_m$, construct disjoint open intervals $\mathcal{I}_1$ and $\mathcal{I}_2$ around $p_1$ and $p_2$ respectively, and argue from contiguity, that $\mathbb{Q}(\hat{p}(\bm X) \in \mathcal{I}_i) \rightarrow 1$ for $i=1,2$, a contradiction!
\end{remark}

The question now, is that how do the sets $\Theta_m$ look like for different values of $m\in [0,1)$? To answer this, let us define for each $(\beta,p)$,
$$\beta^*(p) := \sup \left\{\beta >0: \sup_{x\in [0,1]}  H_{\beta,p}(x) = 0\right\}~.$$
It follows from \cite{jaesungcw} that for $p\in D$, the function $H_{\beta,p}$ has a unique postitive maximizer $m_*(\beta,p)$ if $\beta > \beta^*(p)$. By convention, we define $m_p = m_*(\beta^*(p),p) := \lim_{\beta \rightarrow \beta^*(p)^-} m_*(\beta,p)$.

\begin{prop}\label{thetamsize}
   The sequence $\{m_p\}_{p\in D} \rightarrow 1$ as $p\rightarrow \infty$, and can be sorted in ascending order as $0<m_{p_1}\le m_{p_2}\le \ldots < 1$. Further, if we denote $\Theta_m\Big|_2$ to be the projection of $\Theta_m$ onto the $p$-coordinate, then

   \[   
\Theta_m\Big|_2 = 
     \begin{cases}
       \emptyset &\quad\text{if}~0<m\le m_{p_1},\\
       \{p_1,\ldots,p_k\} &\quad\text{if}~ m_{p_k} < m \le m_{p_{k+1}},~k\ge 1.\\
     \end{cases}
\]
\end{prop}

The proof of Proposition \ref{thetamsize} is technical, and is given in Appendix \ref{sec:apA}. Note that for $m\in (0,1)$, the set $\Theta_m$ is uniquely determined by its projection $\Theta_m\Big|_2$, because every $p\in \Theta_m\Big|_2$ uniquely corresponds to the element $(\beta_p,p)\in \Theta_m$, where $\beta_p := p^{-1}m^{1-p}\tanh^{-1}(m)$. Proposition \ref{thetamsize} thus gives a complete description of the family of sets $\Theta_m$ for all values of $m\in (0,1)$, and says that although each of these sets is finite, we can choose sets as large as possible from this family. The following proposition describes the set $\Theta_0$.

\begin{prop}\label{mzeroprop}
    $\Theta_0 = \{(\beta,p)\in (0,\infty)\times D: \beta < \beta^*(p)\}$.
\end{prop}

The proof of Proposition \ref{mzeroprop} follows from the following three facts proved in \cite{jaesungcw}:
\begin{enumerate}
    \item If $\beta<\beta^*(p)$, then $0$ is the unique global maximizer of $H_{\beta,p}$.
    \item If $\beta =\beta^*(p)$, then $H_{\beta.p}$ has two different non-negative global maximizers.
    \item If $\beta>\beta^*(p)$, then $H_{\beta,p}$ has a unique non-negative global maximizer, which happens to be positive.
\end{enumerate}

\begin{remark}
	All the results in this section also hold almost surely for the somewhat more general $p$-spin Erd\H{o}s-R\'enyi Ising model \cite{barra2}, where the interaction tensor $\bm J$ in \eqref{genising} is given by $\alpha^{-1} N^{1-p}\bm A$, with the entries $A_{i_1,\ldots,i_p}$ of the tensor $\bm A$ being i.i.d. Bernoulli random variables with mean $\alpha$, for some fixed $\alpha\in (0,1)$. This follows from the fact that the Erd\H{o}s-R\'enyi Ising and Curie-Weiss measures are mutually contiguous, which follows from Lemma 6.6 in \cite{Bahadursm}. 
	\end{remark}

\section{Estimation of $p$ when $\beta$ is Known}
Throughout this section, we will assume that the parameter $\beta$ is known. To begin with, for every $\beta >0$, let us define the following two sets:
$$U_\beta :=\{p\ge 2:\beta^*(p) > \beta\}\quad\text{and}\quad L_\beta :=\{p\ge 2:\beta^*(p) < \beta\}.$$
 By Lemma A.1 in \cite{jaesung1}, the set $U_\beta$ is empty if $\beta \ge \log 2$, and is of the form $\{q,q+1,\ldots\}$ for some integer $q\ge 2$ otherwise.
\begin{theorem}\label{pmp}
    When $\beta$ is known, there does not exist any sequence of estimators which is consistent for $p \in U_\beta$.
\end{theorem}

\begin{proof}
    It follows from Proposition \ref{mzeroprop} and Lemma \ref{contig}, that the $N$-fold product measure of the mean-$0$ Rademacher distribution is contiguous to $\p_{\beta,p}$ for all $p\in U_\beta$. The proof now follows from the argument given in Remark \ref{rem1}.
\end{proof}

\begin{remark}
    One can consider an extension of the model \eqref{cwm} by adding an external magnetic field parameter $h$ as follows:
    \begin{equation}\label{cwmh}
     \p_{\beta,p}(\bm x) := \frac{\exp\left\{\beta N \bar{x}^p + h N \bar{x}\right\}}{Z_N(\beta,p)} \quad \text{for}~\bm x\in \{-1,1\}^n
 \end{equation}
    It is shown in \cite{jaesungcw} that consistent estimation of $h$ is always possible when $\beta$ is known. Proposition \ref{pmp} shows that this is not the case for estimating $p$, which is impossible if $\beta^*(p)>\beta$. This inestimability region also coincides with that for estimating $\beta$ when $p$ is known (see \cite{jaesung1}).
\end{remark}

Turning our attention now to the set $L_\beta$, we break it down to the following two parts:
$$L_\beta^1 := \{p\in L_\beta: \exists ~q\in L_\beta,~ q\ne p~\text{such that}~m_*(\beta,q)=m_*(\beta,p)\}\quad\text{and}\quad L_\beta^2 := L_\beta \setminus L_\beta^1~.$$
where $m_*(\beta,q)$ is defined as the largest non-negative maximizer of $H_{\beta,q}$. 

\begin{theorem}
     When $\beta$ is known, there does not exist any sequence of estimators which is consistent for $p \in L_\beta^1$.
\end{theorem}

\begin{proof}
    Once again, by Lemma \ref{contig}, we know that the $N$-fold product of the Rademacher distribution with mean $m_*(\beta,p)$ is contiguous to both the measures $\p_{\beta,p}$ and $\p_{\beta,q}$, where $p\in L_\beta^1$ and $p\ne q\in L_\beta$ is such that $m_*(\beta,q)=m_*(\beta,p)$. Clearly, $q\in L_\beta^1$. Once again, the rest of the proof follows from the arguments of Remark \ref{rem1}. 
\end{proof}

We now show that it is possible to estimate $p$ consistently on the set $L_\beta^2$, if $\beta$ is known. Towards this, let us define the following estimator of $p$:

\begin{equation}\label{algo}
    \hat{p}(\beta,\delta) := \argmin_{2\le q\le \log_{|\bar{X}|^{-1}} (2\beta + \delta) +2~,~\beta^*(q)<\beta} \left|m_*(\beta,q)^2 - \bar{X}^2\right|
\end{equation}
 The intuition behind constructing our estimator, is that $\bar{X}^2 \xrightarrow{P} m_*(\beta,p)^2$ as $N\rightarrow \infty$, and hence, the distance between $\bar{X}^2$ and $m_*(\beta,q)^2$ is expected to be minimized at $q=p$ for large $N$. The next theorem shows that with high probability, this is indeed the case.

\begin{theorem}\label{constestp}
   For every $\delta >0$ and $p\in L_\beta^2$, we have 
   $$\p_{\beta,p}(\hat{p}_{\beta,\delta} = p) = 1-e^{-CN}$$ for some constant $C = C_{\beta,p,\delta}>0$ not depending on $N$.
\end{theorem}

\begin{proof}
    Suppose that $\bm X \sim \p_{\beta,p}$ for some $p\in L_\beta^2$. It follows from Lemma A.1 in \cite{jaesung1}, and the proof of Proposition \ref{thetamsize} that as $q\rightarrow \infty$, $m_*(\beta,q)\rightarrow 1$ if $\beta \ge \log 2$ and $m_*(\beta,q)\rightarrow 0$ otherwise. Hence, in any case $m_*(\beta,p)\in (0,1)$ is not an accumulation point of the sequence $\{m_*(\beta,q)\}_{q\ge 2}$. This shows that there exists $\varepsilon >0$ such that $|m_*(\beta,q)^2-m_*(\beta,p)^2| >\varepsilon$ for all $q\in L_\beta$ not equal to $p$.

    Now, it follows from the proofs of Lemma 3.1 and Lemma 3.3 in \cite{jaesungcw} that $$\p_{\beta,p}\left(|m_*(\beta,p)^2 -\bar{X}^2| >\frac{\varepsilon}{3}\right) \le e^{-C_1 N}$$ for some constant $C_1 >0$ not depending on $N$. Also, it follows from the proofs of Lemma 3.1 and Lemma 3.3 in \cite{jaesungcw} and the proof of Proposition \ref{thetamsize}, that
    $$\p_{\beta,p}\left(\bar{X}^2 < (2\beta +\delta)^{2/(2-p)}\right) \le e^{-C_2N}\quad \text{i.e.}\quad \p_{\beta,p}\left(p> \log_{|\bar{X}|^{-1}} (2\beta + \delta) +2\right) \le e^{-C_2N}$$ for some constant $C_2 >0$ not depending on $N$. Define $E_1 := \{|m_*(\beta,p)^2 -\bar{X}^2| \le\varepsilon/3\}$ and $E_2 := \{p\le \log_{|\bar{X}|^{-1}} (2\beta + \delta) +2\}$. It is clear that $\hat{p}(\beta,\delta)=p$ on the event $E_1\bigcap E_2$. Theorem \ref{constestp} now follows, since $\p(E_1^c \bigcup E_2^c) \le e^{-C_1N}+e^{-C_2N}$.
\end{proof}
\begin{remark}
    One can tune the parameter $\delta>0$ in \eqref{algo} to allow an optimal number of integers over which the minimization is to be done, in order to ensure proper convergence of $\hat{p}(\beta,\delta)$. Also, in case $\beta<\log 2$, eventually $\beta^*(q)>\beta$, so one just stops at the largest $q$ for which $\beta^*(q) <\beta$ in the minimization \eqref{algo}.
\end{remark}

In the reverse problem of estimating $\beta$ with known $p$, consistent estimation is possible for all $\beta >\beta^*(p)$ (see \cite{jaesungcw}). In contrast to this, a further annoying inestimability region $L_\beta^1$ arises in our problem of estimating $p$ with known $\beta$. The next proposition shows that for almost all $\beta>0$ the set $L_{\beta}^1$ is actually empty, i.e. the entire region $L_\beta$ is estimable.

\begin{prop}\label{annoy}
    There exists a countable set $\Xi \subset (0,\infty)$, such that $L_\beta^1 = \emptyset$ for $\beta \notin \Xi$.
\end{prop}
The proof of Proposition \ref{annoy} is given in Appendix \ref{proofannoy}. Although in the proof of Proposition \ref{annoy} we give an exact enumeration of one countable set $\Lambda$ satisfying Proposition \ref{annoy}, namely:
$$\Lambda :=\{s_{p,q}: q> p \ge 2~\text{are integers}\}~\quad \text{where}\quad s_{p,q} := \frac{1}{p}\left(\frac{p}{q}\right)^{\frac{1-p}{q-p}}\tanh^{-1}\left(\left(\frac{p}{q}\right)^{\frac{1}{q-p}}\right),$$
 our mathematical skills did not allow us to dig in any further. In particular, the following question is open:
\vspace{0.1in}

\noindent \textbf{Open Problem:} Give an exact enumeration of the minimal countable set $\Xi_0 = \{\beta >0: L_\beta^1\ne \emptyset\}$ satisfying Proposition \ref{annoy}.
\vspace{0.1in}

Note that if $\beta = s_{p,q}$ for some integers $q>p\ge 2$, then $m:=(p/q)^{1/(q-p)}$ is a stationary point of both $H_{\beta,p}$ and $H_{\beta,q}$. We can actually say that $\beta \in \Xi_0$ if this stationary point turns out to be a global maximizer of both these functions. However, checking the last condition will likely involve more intricate analysis, and is left as an open question.

\section{Acknowledgement}
\noindent S. M. was supported by the FoS Tier 1 grant WBS A-8001449-00-00. The author thanks Paul Switzer for bringing to notice the problem of estimating $p$, and Luc Devroye for interesting discussions related to this problem.

 \appendix
 \section{Proof of Proposition \ref{thetamsize}}\label{sec:apA}
To show that the sequence $m_p \rightarrow 1$ as $p\rightarrow \infty$, note that by a Taylor expansion, we have:
$$H_{\beta,p}(x) \le \beta x^p - \frac{x^2}{2}\quad\text{for all}~\beta >0~.$$ Hence, $H_{\beta,p}(x) < 0$ for all $x < (2\beta)^{1/(2-p)}$, and so, if $\beta > \beta^*(p)$, then $m_*(\beta,p) \ge (2\beta)^{1/(2-p)}$. Thus, for every $\beta \in (\beta^*(p),\beta^*(p)+\varepsilon)$, we have
$m_*(\beta,p) > (2(\beta^*(p)+\varepsilon))^{1/(2-p)}$. Since $\varepsilon >0$ is arbitrary, we thus have:
\begin{equation}\label{aboveinq}
  m_p \ge (2\beta^*(p))^{1/(2-p)}  
\end{equation}
It is now easy to see that the right-hand side of \eqref{aboveinq} converges to $1$ as $p\rightarrow \infty$, whereas the left-hand side is bounded above by $1$. This proves our claim. 

Now, since the sequence $\{m_p\}_{p\in D}$ is convergent to its supremum, it can be arranged in increasing order as $m_{p_1} \le m_{p_2} \le \ldots < 1$. Also, by \eqref{aboveinq} and Lemma A.1 in \cite{jaesung1}, $m_{p_1} \ge (\log 4)^{1/(2-p_1)}>0$. Suppose that $(\beta,p)\in \Theta_m$ for some $m>0$. It follows from \cite{jaesungcw} that $\beta >\beta^*(p)$. By Lemma 11 in \cite{Bahadursm}, the function $m_*(\beta,p)$ is strictly increasing in $\beta$ on $(\beta^*(p),\infty)$, and hence,
\begin{equation}\label{mstar}
    m=m_*(\beta,p) > m_p~.
\end{equation}
It follows from \eqref{mstar}, that $\Theta_m=\emptyset$ for $m\in (0,m_{p_1}]$.

Let us now assume that $m\in (m_{p_k},m_{p_{k+1}}]$. If $(\beta,p)\in \Theta_m$, then by \eqref{mstar}, we have $m_{p_{k+1}} > m_p$ and hence, $p\in \{p_1,\ldots,p_k\}$. Conversely, suppose that $p=p_i$ for some $1\le i\le k$. Then, $m>m_p$ and hence, there exists $\beta_0>\beta^*(p)$ such that $m>m_*(\beta_0,p)$. By Lemma B.3 in \cite{Bahadursm}, we know that $\lim_{\beta\rightarrow\infty} m_*(\beta,p)=1$, which implies the existence of $\beta_1>\beta^*(p)$ such that $m<m_*(\beta_1,p)$. Finally, Lemma B.3 in \cite{Bahadursm} also states that the function $m_*(\beta,p)$ is continuous in $\beta$ on $(\beta^*(p),\infty)$, which in view of the intermediate value theorem, implies the existence of $\beta\in (\beta_0,\beta_1)$, such that $m=m_*(\beta,p)$, i.e. $(\beta,p) \in \Theta_m$. This completes the proof of Proposition \ref{thetamsize}. 

\section{Proof of Proposition \ref{annoy}}\label{proofannoy}
For every pair of integers $q>p\ge 2$, define:
$$s_{p,q} := \frac{1}{p}\left(\frac{p}{q}\right)^{\frac{1-p}{q-p}}\tanh^{-1}\left(\left(\frac{p}{q}\right)^{\frac{1}{q-p}}\right)~,$$ and let $\Lambda := \{s_{p,q}: q>p\ge 2~\text{are integers}\}$. Clearly, $\Lambda$ is a countable set. Suppose that $L_{\beta}^1\ne \emptyset$. Then, there exist $q>p\ge 2$ such that both the functions $H_{\beta,p}$ and $H_{\beta,q}$ are maximized at the same point $m\in (0,1)$. Hence, we have $H_{\beta,q}'(m)=H_{\beta,p}'(m) =0$, i.e.
$$\beta q m^{q-1} = \beta pm^{p-1} = \tanh^{-1}(m)~.$$
From the first equality, we have $m=(p/q)^{1/(q-p)}$, and from the last equality, we have:
$$\beta = \frac{\tanh^{-1}(m)}{pm^{p-1}} = s_{p,q} \in \Lambda~.$$
Hence, $L_{\beta}^1 = \emptyset$ if $\beta\notin \Lambda$, which completes the proof of Proposition \ref{annoy}.




\end{document}